\documentclass[11pt]{amsart}
\usepackage{amscd,amssymb,latexsym}
\usepackage{fullpage}
\newcommand{\Mdef}[2]{\newcommand{#1}{\relax \ifmmode #2 \else $#2$\fi}}

\newcommand{\tensor}{\otimes}

\newcommand{\sdr}{\rtimes}
\newcommand{\map}{\mathrm{map}}
\newcommand{\Hom}{\mathrm{Hom}}

\newcommand{\Ext}{\mathrm{Ext}}
\Mdef{\bhom}{\mathbf{\hat{H}om}}

\Mdef{\Mod}{\mathrm{mod}}

\newcommand{\st}{\; | \;}

\newtheorem{thm}{Theorem}[section]
\newtheorem{lemma}[thm]{Lemma}
\newtheorem{prop}[thm]{Proposition}
\newtheorem{cor}[thm]{Corollary}

\theoremstyle{definition}

\newtheorem{defn}[thm]{Definition}

\newtheorem{assumption}[thm]{Assumption}

\newtheorem{example}[thm]{Example}

\newtheorem{remark}[thm]{Remark}

\newcommand{\qqed}{\qed \\[1ex]}
\renewenvironment{proof}[1][\hspace*{-.8ex}]{\noindent {\bf Proof #1:\;}}{\qqed}

\Mdef{\PH} {\Phi^H}
\Mdef{\PK} {\Phi^K}
\Mdef{\PL} {\Phi^L}
\Mdef{\PT} {\Phi^{\T}}

\Mdef{\ef}{E{\cF}_+}
\Mdef{\etf}{\widetilde{E}{\cF}}
\Mdef{\eg}{E{G}_+}
\Mdef{\etg}{\tilde{E}{G}}

\Mdef{\infl}{\mathrm{inf}}
\Mdef{\defl}{\mathrm{def}}
\Mdef{\res}{\mathrm{res}}
\Mdef{\ind}{\mathrm{ind}}
\Mdef{\coind}{\mathrm{coind}}

\Mdef{\univ}{\mathcal{U}}

\Mdef{\Fp}{\mathbb{F}_p}
\Mdef{\Zpinfty}{\Z /p^{\infty}}
\Mdef{\Zpadic}{\Z_p^{\wedge}}

\newcommand{\dichotomy}[2]{\left\{ \begin{array}{ll}#1\\#2 \end{array}\right.}

\newcommand{\lra}{\longrightarrow}
\newcommand{\lla}{\longleftarrow}

\newcommand{\spec}{\mathrm{Spec}}

\Mdef{\we}{\mathbf{we}}
\Mdef{\fib}{\mathbf{fib}}
\Mdef{\cof}{\mathbf{cof}}
\Mdef{\BI}{\mathcal{BI}}

\newcommand{\colim}{\mathop{  \mathop{\mathrm {lim}} \limits_\rightarrow} \nolimits}

\newcommand{\hocolim}{\mathop{  \mathop{\mathrm {holim}}\limits_\rightarrow} \nolimits}

\Mdef{\B}{\mathbb{B}}
\Mdef{\C}{\mathbb{C}}
\Mdef{\D}{\mathbb{D}}
\Mdef{\E}{\mathbb{E}}
\Mdef{\T}{\mathbb{T}}
\Mdef{\F}{\mathbb{F}}
\Mdef{\G}{\mathbb{G}}
\Mdef{\I}{\mathbb{I}}
\Mdef{\N}{\mathbb{N}}
\Mdef{\Q}{\mathbb{Q}}
\Mdef{\R}{\mathbb{R}}
\Mdef{\bbS}{\mathbb{S}}
\Mdef{\Z}{\mathbb{Z}}

\Mdef{\bA}{\mathbb{A}}
\Mdef{\bB}{\mathbb{B}}
\Mdef{\bC}{\mathbb{C}}
\Mdef{\bD}{\mathbb{D}}
\Mdef{\bE}{\mathbb{E}}
\Mdef{\bF}{\mathbb{F}}
\Mdef{\bG}{\mathbb{G}}
\Mdef{\bH}{\mathbb{H}}
\Mdef{\bI}{\mathbb{I}}
\Mdef{\bJ}{\mathbb{J}}
\Mdef{\bK}{\mathbb{K}}
\Mdef{\bL}{\mathbb{L}}
\Mdef{\bM}{\mathbb{M}}
\Mdef{\bN}{\mathbb{N}}
\Mdef{\bO}{\mathbb{O}}
\Mdef{\bP}{\mathbb{P}}
\Mdef{\bQ}{\mathbb{Q}}
\Mdef{\bR}{\mathbb{R}}
\Mdef{\bS}{\mathbb{S}}
\Mdef{\bT}{\mathbb{T}}
\Mdef{\bU}{\mathbb{U}}
\Mdef{\bV}{\mathbb{V}}
\Mdef{\bW}{\mathbb{W}}
\Mdef{\bX}{\mathbb{X}}
\Mdef{\bY}{\mathbb{Y}}
\Mdef{\bZ}{\mathbb{Z}}

\Mdef{\cA}{\mathcal{A}}
\Mdef{\cB}{\mathcal{B}}
\Mdef{\cC}{\mathcal{C}}
\Mdef{\mcD}{\mathcal{D}} 
\Mdef{\cE}{\mathcal{E}}
\Mdef{\cF}{\mathcal{F}}
\Mdef{\cG}{\mathcal{G}}
\Mdef{\mcH}{\mathcal{H}} 
\Mdef{\cI}{\mathcal{I}}
\Mdef{\cJ}{\mathcal{J}}
\Mdef{\cK}{\mathcal{K}}
\Mdef{\mcL}{\mathcal{L}}

\Mdef{\cM}{\mathcal{M}}
\Mdef{\cN}{\mathcal{N}}
\Mdef{\cO}{\mathcal{O}}
\Mdef{\cP}{\mathcal{P}}
\Mdef{\cQ}{\mathcal{Q}}
\Mdef{\mcR}{\mathcal{R}}
\Mdef{\cS}{\mathcal{S}}
\Mdef{\cT}{\mathcal{T}}
\Mdef{\cU}{\mathcal{U}}
\Mdef{\cV}{\mathcal{V}}
\Mdef{\cW}{\mathcal{W}}
\Mdef{\cX}{\mathcal{X}}
\Mdef{\cY}{\mathcal{Y}}
\Mdef{\cZ}{\mathcal{Z}}

\Mdef{\ca}{\mathcal{a}}
\Mdef{\ct}{\mathcal{t}}

\Mdef{\At}{\tilde{A}}
\Mdef{\Bt}{\tilde{B}}
\Mdef{\Ct}{\tilde{C}}
\Mdef{\Et}{\tilde{E}}
\Mdef{\Ht}{\tilde{H}}
\Mdef{\Kt}{\tilde{K}}
\Mdef{\Lt}{\tilde{L}}
\Mdef{\Mt}{\tilde{M}}
\Mdef{\Nt}{\tilde{N}}
\Mdef{\Pt}{\tilde{P}}

\Mdef{\tA}{\tilde{A}}
\Mdef{\tB}{\tilde{B}}
\Mdef{\tC}{\tilde{C}}
\Mdef{\tE}{\tilde{E}}
\Mdef{\tH}{\tilde{H}}
\Mdef{\tK}{\tilde{K}}
\Mdef{\tL}{\tilde{L}}
\Mdef{\tM}{\tilde{M}}
\Mdef{\tN}{\tilde{N}}
\Mdef{\tP}{\tilde{P}}

\Mdef{\ft}{\tilde{f}}
\Mdef{\xt}{\tilde{x}}
\Mdef{\yt}{\tilde{y}}

\Mdef{\Ab}{\overline{A}}
\Mdef{\Bb}{\overline{B}}
\Mdef{\Cb}{\overline{C}}
\Mdef{\Eb}{\overline{E}}
\Mdef{\Fb}{\overline{F}}
\Mdef{\Gb}{\overline{G}}
\Mdef{\Hb}{\overline{H}}
\Mdef{\Ib}{\overline{I}}
\Mdef{\Jb}{\overline{J}}
\Mdef{\Kb}{\overline{K}}
\Mdef{\Lb}{\overline{L}}
\Mdef{\Mb}{\overline{M}}
\Mdef{\Nb}{\overline{N}}
\Mdef{\Ob}{\overline{O}}
\Mdef{\Pb}{\overline{P}}
\Mdef{\Qb}{\overline{Q}}
\Mdef{\Rb}{\overline{R}}
\Mdef{\Sb}{\overline{S}}
\Mdef{\Tb}{\overline{T}}
\Mdef{\Ub}{\overline{U}}
\Mdef{\Vb}{\overline{V}}
\Mdef{\Wb}{\overline{W}}
\Mdef{\Xb}{\overline{X}}
\Mdef{\Yb}{\overline{Y}}
\Mdef{\Zb}{\overline{Z}}

\Mdef{\db}{\overline{d}}
\Mdef{\hb}{\overline{h}}
\Mdef{\qb}{\overline{q}}
\Mdef{\rb}{\overline{r}}
\Mdef{\tb}{\overline{t}}
\Mdef{\ub}{\overline{u}}
\Mdef{\vb}{\overline{v}}

\Mdef{\hc}{\hat{c}}
\Mdef{\he}{\hat{e}}
\Mdef{\hf}{\hat{f}}
\Mdef{\hA}{\hat{A}}
\Mdef{\hH}{\hat{H}}
\Mdef{\hJ}{\hat{J}}
\Mdef{\hM}{\hat{M}}
\Mdef{\hP}{\hat{P}}
\Mdef{\hQ}{\hat{Q}}

\Mdef{\thetab}{\overline{\theta}}
\Mdef{\phib}{\overline{\phi}}

\Mdef{\uA}{\underline{A}}
\Mdef{\uB}{\underline{B}}
\Mdef{\uC}{\underline{C}}
\Mdef{\uD}{\underline{D}}

\Mdef{\bolda}{\mathbf{a}}
\Mdef{\boldb}{\mathbf{b}}
\Mdef{\bfD}{\mathbf{D}}

\Mdef{\fm}{\frak{m}}
\Mdef{\fp}{\frak{p}}

\Mdef{\eps}{\epsilon}

\newcommand{\cell}{\mathrm{Cell}}

\renewcommand{\cC}{\mathcal{C}}

\newcommand{\sfD}{\mathsf{D}}

\newcommand{\Db}{\sfD^b}

\newcommand{\Dsg}{\sfD_{sg}}
\newcommand{\Dcsg}{\sfD_{csg}}
\newcommand{\CBG}{C^*(BG)}
\newcommand{\ChatBG}{\hat{C}^*(BG)}
\newcommand{\HBG}{H^*(BG)}
\newcommand{\COBG}{C_*(\Omega BG_p)}
\newcommand{\HOBG}{H_*(\Omega BG_p)}
\newcommand{\CCOBG}{C^*(\Omega BG_p)}
\newcommand{\TCOBG}{L_kC_*(\Omega BG_p)}
\newcommand{\CBU}{C^*(BU)}

\newcommand{\COBU}{C_*(U_p)}

\newcommand{\COX}{C_*(\Omega X)}
\newcommand{\HOX}{H_*(\Omega X)}
\newcommand{\TCOX}{L_kC_*(\Omega X)}
\newcommand{\COF}{C_*(\Omega F)}

\newcommand{\TCOF}{L_k C_*(\Omega F)}
\newcommand{\COUG}{C_*(\Omega U/G_p)}
\newcommand{\TCOUG}{L_kC_*(\Omega U/G_p)}

\newcommand{\Ho}{\mathrm{Ho}}
\newcommand{\OBG}{\Omega BG_p}
\newcommand{\modules}{\mbox{-mod}}
\newcommand{\nubar}{\overline{\nu}}

\newcommand{\HO}{H^{\Omega}}
\newcommand{\HOC}{H_{\Omega}}
\newcommand{\HOT}{\hat{H}^{\Omega}}
\newcommand{\tauz}{\tau_{\geq 0}}
\newcommand{\taus}{\tau_{\geq s}}

\newcommand{\Loc}{\mathrm{Loc}}

\newcommand{\Perm}{\mathrm{Perm}}

\newcommand{\DPerm}{\sfD (\Perm (G,k))}
\input{xypic}

\begin{document}
\title{The singularity category and duality for  complete intersection
  groups}
\author{J.P.C.Greenlees}
\address{Mathematics Institute, Zeeman Building, Coventry CV4, 7AL, UK}
\email{john.greenlees@warwick.ac.uk}

\date{}

\begin{abstract}
If $G$ is a finite group,  the structure of the modular representation
theory depends on the cochains $C^*(BG; k)$, viewed as a commutative
ring spectrum. We consider here its singularity category (in the sense
of the author and Stevenson \cite{Dsg}) and show that  the
singularity category is the bounded derived category of the
$\Omega$-Tate spectrum (the $k$-nullification of the Koszul dual connective ring
spectrum $\COBG$). We establish a form of
Gorenstein duality for $\COBG$, and show that the $\Omega$-Tate
spectrum enjoys a form of Tate duality. Under a  complete
intersection hypothesis we give a method for calculating the
$\Omega$-Tate homology
\end{abstract}

\thanks{   The author is grateful to Dave Benson for conversations about
  this work, especially Lemma \ref{lem:Opzero}. This research was supported by EPSRC grant EP/W036320/1.   The author
  would also  like to thank the Isaac Newton   Institute for
  Mathematical Sciences, Cambridge, for support and   hospitality during the programme Equivariant Homotopy Theory in  Context, where later parts of  work on this paper was undertaken. This work was supported by EPSRC grant EP/Z000580/1.  } 
\maketitle

\setcounter{tocdepth}{1}
\tableofcontents

\section{Introduction}

This paper is a contribution to the  overarching aim of describing a hierarchy of behaviour for finite
groups $G$ from the point of view of commutative algebra (we also
comment on compact Lie groups $G$). 
We recall the context in the rest of the introduction, but
highlight here that the novelty here is based around 
the $\Omega$-Tate homology. We show that it
 satisfies a form  of Tate duality and that it 
 captures the singularity category of $C^*(BG)$. Finally, if the group
 satisfies a complete intersection condition (for example Chevalley groups at
good primes),  we give methods for calculating the $\Omega$-Tate homology.

\subsection{The enhanced group cohomology ring}
Many structural features of the representation theory of a finite group $G$
over a field $k$ of characteristic $p$ are reflected in the cohomology ring
$H^*(BG;k)=\Ext_{kG}^*(k,k)$, starting with Quillen's theorem that the
Krull dimension is the $p$-rank of $G$. This is a Noetherian ring (Venkov) and  
very special structurally: for example if it is Cohen-Macaulay, it is
automatically Gorenstein (Benson-Carlson \cite{BC}). However the structural
features are more clearly reflected if we consider an enrichment: we
consider the cochains $C^*(BG)=C^*(BG;k)$ rather than the cohomology
ring  $H^*(BG)=\pi_*(C^*(BG))$. For many purposes  it is
enough to consider it as an $A_\infty$-ring, which is familiar in
algebraic contexts, but in fact we may take $C^*(BG)=\map (BG, Hk)$ to
be the spectrum of maps from $BG$ into the Eilenberg-Maclane spectrum
$Hk$. As such it is an $E_{\infty}$-ring, and we may capture relevant
structures by working in a symmetric monoidal category of spectra in
which it is a commutative ring.

\subsection{The spectrum of behaviour}
A massive benefit of working with cochains is that $C^*(BG)$ is
Gorenstein \cite[Subsection 10.3]{DGI} for all finite groups $G$ without exception. At
the other extreme, following Auslander, Buchsbaum and Serre in
classical commutative algebra, one may define regular
local rings in a homotopy invariant way, and it turns out that $C^*(BG)$ is regular if
and only if $G$ is $p$-nilpotent. 

It is then natural to consider the
spread of behaviour on the spectrum between 
Gorenstein and regular, and to use the singularity category to place
groups along the range. 

Results of \cite{Dsg} (explained below) allow the apparatus of the
singularity category to be applied for $\CBG$. Some specific
calculations have been made in \cite{Dsgpq, BensonDsg}. 
 In these cases it was possible to calculate all coefficient rings and 
 to give small and explicit algebraic models, but we cannot expect to be explicit in 
 general. In the present paper we develop some  structural and
 homotopy invariant methods we can apply more generally.

In particular, we prove that the singularity category is
 the bounded derived category of  the $\Omega$-Tate ring
 spectrum (Definition \ref{defn:OmegaTate}). When  $C^*(BG)$ is a complete intersection
in a suitable sense, we give methods of calculation.

\subsection{Koszul duality}
Morita theory allows us to take a 
$kG$-module $M$ and obtain the module $C^*(BG; M):=\Hom_{kG}(k, M)$
over the Koszul dual ring 
$(kG)^!:=\Hom_{kG}(k,k)\simeq C^*(BG)$. One might hope this is one direction
of  a Morita equivalence, but the situation is a little more complicated: if we attempt to return to
$kG$-modules we obtain an action of the ring $\cE$ of 
$C^*(BG)$-endomorphisms  of $k$. 
The Eilenberg-Moore spectral sequence arises from an equivalence
$$\Hom_{C^*X} (k,k)\simeq C_*(\Omega X), $$
provided $X$ is connected, $p$-complete and $\pi_1(X)$ is a finite
$p$-group \cite{DwyerStrong}. Since
the Bousfield-Kan $p$-completion 
$BG\lra (BG)_p^{\wedge}$ induces an isomorphism in $H^*(\cdot ; k)$,
we see
$$\cE=\CBG^!=\Hom_{C^*(BG)}(k,k)\simeq C_*(\Omega (BG_p^{\wedge})). $$
For brevity we write $\COBG=C_*(\Omega (BG_p^{\wedge}))$ from now on. 
The point is that $\COBG=((kG)^!)^!$ is the double Koszul dual of $kG$, so we
have a double-centralizer completion map  $kG\lra \COBG$. This is an
equivalence if $G$ is a $p$-group, but generally very far from it. For
example if $G$ is not $p$-nilpotent,
the homology ring $\HOBG$ is not finite dimensional. 

\subsection{Morita equivalence}
\label{subsec:Morita}
The advantage of working with $\COBG$ is that we do get a precise
Morita equivalence between appropriate categories of 
$\CBG$-modules and $\COBG$-modules. To establish the context, we must
establish a good theory of `finitely generated' modules. We explain
this in more detail below, but there is a quick definition sufficient
to let us to state our main results. We may say that a $\CBG$-module $M$ is
{\em finitely generated} if $\pi_*(M)$ is finitely generated over
$\pi_*(\CBG)=\HBG$, and define the bounded derived
category to be the homotopy category of modules with homotopy finitely
generated over $\HBG$, 
$$\Db (\CBG)=\Ho (\{ M \st \pi_*(M) \mbox{ is finitely generated over }
H^*(BG) \} )$$
It is shown in \cite{Dsg} that this has good formal properties, and that there is a Morita equivalence
$$\Db (\CBG)\simeq \Db (\COBG). $$
Accordingly, we can move back and forth between $\CBG$-modules and
$\COBG$-modules, which  makes precise what one learns about $kG$-modules by
considering $\CBG$-modules.

\subsection{Gorenstein duality}
We have already highlighted the Gorenstein duality property for
$\CBG$. This is closely related to the fact that there are two
approaches to Tate duality: one which splices together homology and
cohomology and one which kills finite free spectra. The fact that
these two give the same answer is important: it is the Anderson
self-duality of the Tate spectrum. In group cohomology,  this is often
called {\em Tate duality}. 

It is known that the noncommutative ring $\COBG$ is Gorenstein, and
in Theorem \ref{thm:GorDCOBG} we formulate and prove the corresponding  Gorenstein {\em
  duality} statement. 

This in turn is sufficient to allow us to define the $\Omega$-Tate ring
spectrum, and give a description of the singularity category for 
$\CBG$ as the bounded derived category of the $\Omega$-Tate 
spectrum (Corollary \ref{cor:DsgisDb}). 

\subsection{Complete intersections}
Finally, under a complete intersection hypotheses, the duality properties for 
$\COBG$ can be formulated as a local cohomology theorem, 
showing that $\HOBG$ is a very special ring (Corollary
\ref{cor:HHciiswc} and Section \ref{sec:examples}).

To explain the assumption, recall that in the commutative
algebra of Noetherian rings Gulliksen has
shown that ci rings are precisely those for which the Ext algebra
$\Ext_R^*(k,k)$ has polynomial growth. Without an 
assumption of that type, the ring $\HOBG\simeq
\pi_*(\Hom_{\CBG}(k,k))$ has no hope of good Noetherian
behaviour. Accordingly it is entirely reasonable to make a ci assumption. It
was shown in \cite{pzci} that there is a range of different ways to
transpose the ci assumption to ring spectra like $\CBG$, and we will
assume a form of the ci condition. Section \ref{sec:centralred} 
shows that under this assumption,
the localization can be constructed using  a stable Koszul complex
 inverting central elements, which gives a convenient
method of calculation. 

\subsection{Organization}
In Section \ref{sec:HICA} we describe the general context of homotopy
invariant commutative algebra, and in Section \ref{sec:normsymm}  we 
focus in on the symmetric Gorenstein
context obtained from a faithful representation of the group in a
connected compact Lie group.
In Section \ref{sec:Tate} we introduce the $\Omega$-Tate spectrum as a
Bousfield localization, and in Section \ref{sec:AndOmega} we show that it
enjoys Anderson duality for $k$-orientable groups. In Section
\ref{sec:nu} we consider the norm map and determine it for all
finite groups. In Section \ref{sec:squeeze} we give an alternative
approach through Benson's squeezed resolutions, hence determining
$\Omega$-Tate homology in purely algebraic terms.
In Section \ref{sec:sing} we recall the definition of the  singularity category of
$\CBG$, and in Section \ref{sec:modovertate} we study modules over the
$\Omega$-Tate spectrum and show that the singularity category of
$\CBG$ is the bounded derived category of the $\Omega$-Tate spectrum.
In Section \ref{sec:centralred} we show that under complete
intersection hypotheses we may calculate the $\Omega$-Tate homology by
a Cech complex. Finally, in Section \ref{sec:examples} we mention a
few examples of groups to which this applies. 

\section{Homotopy invariant commutative algebra and  Morita
  equivalence}
\label{sec:HICA}
The motivation for our methods comes from classical commutative algebra with the study 
of a Noetherian local ring with residue field $k$. In this section we recall some
basic homotopy invariant definitions and how they apply to our examples.

We note the recurrent theme that it is sometimes best to look at modules over
the commutative algebra $\CBG$ and sometimes best to look at modules
over  $\COBG$. 

\subsection{Building}
We recall some useful language from \cite{DGI}.

Given an object $M$ in a triangulated category, we say $M$ {\em finitely
  builds} $N$ if $N$ can be constructed in finitely many steps from
$M$ by taking mapping cones, suspensions and retracts  (in other
words, if $N$ is in the thick subcategory generated by $M$). We say
that $M$ {\em builds} $N$ if arbitrary coproducts are permitted (in
other words, if $N$ is in the localizing subcategory generated by $M$).

\subsection{Regularity}
By results of Auslander, Buchsbaum and Serre, a Noetherian commutative local ring
is regular if and only if $k$ is small as an $R$-module, or
equivalently that $\Ext_R^*(k,k)$ is finite dimensional.   For a ring
spectrum $R$ with a map $R\lra k$ we {\em define} $R$ to be regular if
$k$ is small over $R$.  This is again equivalent to
$\pi_*\cE=\pi_*(\Hom_R(k,k))$ being finite dimensional. 

Thus  $\CBG$ is regular if and only if $\pi_*\COBG=\HOBG$ is finite
dimensional, and this happens if and only if $G$ is $p$-nilpotent
\cite[Theorem 7.3]{BGgeneration}.

\subsection{Proxy regularity}
Of course regularity is a very restrictive condition, and we need a more
inclusive finiteness condition to play the role of the Noetherian
condition.  In classical commutative algebra, we may choose generators of
the maximal ideal and consider the Koszul complex $K$. The existence
of such a complex is the finiteness condition we need. 

For ring spectra  with a map $R\lra k$ with $k$ a field, and we will require the 
finiteness hypothesis that $R$ is {\em proxy-regular} in the sense \cite[4.14]{DGI} that 
there is a small $R$-module $K$ finitely built by $k$ and so that $k$
is built by $K$. 

It is shown in \cite[Subsection 5.7]{DGI} that $\CBG$ is proxy regular
if $G$ is any finite or compact Lie group, and by
\cite[4.17]{DGI} this implies $\COBG$ is proxy regular. The ring 
$C_*(G)$ (the group ring if $G$ is finite) is proxy regular by \cite[5.9]{DGI}.

\subsection{The Gorenstein condition}
We first consider the Gorenstein condition. In classical commutative
algebra there are a number of different characterisations of
Gorenstein local rings, one of which is the condition that $\Ext_R^*(k,R)$ is
1-dimensional.

\begin{defn} \cite[8.1]{DGI} A proxy-regular augmented ring spectrum $R\lra
  k$ is said to be {\em Gorenstein} of shift $a$ if there is an equivalence of
  $R$-modules  $\Hom_R(k, R)\simeq \Sigma^a k$. 
\end{defn}
 
It is shown in \cite[Subsection 10.3]{DGI} that for a finite group
$G$, the ring $\CBG$ is Gorenstein of shift 0, and it follows
from the Morita invariance statement \cite[8.5]{DGI} that $\COBG$ is
also Gorenstein of shift 0.

If $G$ is a compact Lie group with $k$-orientable adjoint
representation (i.e., if the conjugation action of $G$ on the tangent
space at the identity preserves orientation), then the rings $C^*(BG)$
are again Gorenstein of shift equal to the dimension of $G$
\cite[Subsection 10.3]{DGI}.

\subsection{Effective constructibility}
\label{subsec:effcon}
The power of Koszul duality is that it gives very organized
constructions of $k$-cellularization. This is based on the observation
that  for any $R$-module $M$, $\Hom_R(k,M)$ is a $\cE$-module and hence built from $\cE$ and
hence $\Hom_R(k,M)\tensor_{\cE}k$ is built from $k$.

\begin{defn} \cite[4.3]{DGI}
  We say $k$-cellularization is {\em effectively constructible} if the
  evaluation map
  $$\Hom_R(k,M)\tensor_{\cE}k \lra M$$
  is the $k$-cellularization for all $R$-modules $M$.
\end{defn}

It is shown in \cite[4.10]{DGI} that if $R$ is proxy-regular then the
cellularization is effectively constructible.

\subsection{Gorenstein duality}
\label{subsec:GorD}
If $R$ is a  $k$-algebra we may form the Brown-Comenetz dual
$R^{\vee}=\Hom_R(k, R)$, which obviously has the Matlis lifting property
$$\Hom_R(T, R^{\vee})=\Hom_k(T,k). $$

Thus if $R\lra k$ is a $k$-algebra which is  Gorenstein of shift $a$,
we  have equivalences
$$\Hom_R(k,R)\simeq \Sigma^ak \simeq \Hom_R(k, \Sigma^a R^{\vee}). $$
A priori this is only an equivalence of $R$-modules,  but if $\cE$ has
a unique action on $k$, it is an equivalence of $\cE$-modules and we
may apply $\tensor_{\cE}k$ to deduce
$$\cell_kR\simeq \Sigma^a \cell_k(R^{\vee}). $$
The trivial action condition for $\CBG$ and for $\COBG$ is automatic
for finite  groups $G$ (since
$\pi_1(BG_p)$ is a finite $p$-group that can only act trivially on
$k$). For a compact Lie group, the action is given by the action of
$\pi_1(BG_p)$ on $H^d(S^{ad(G)};k)$; if this  is trivial, we say that  the adjoint
representation is $k$-orientable (this is automatic if $G$ is finite
or connected or if $k$ is of characteristic 2). 

Finally, if $R$ is connective or coconnective with $R_0=k$ a field, we
see that $R^{\vee}$ is already $k$-cellular, so that we have the
Gorenstein duality statement
$$\cell_kR\simeq \Sigma^aR^{\vee}. $$
Under complete intersection  hypotheses we will also give an algebraic description of
$\cell_kR$ putting the Gorenstein duality into the form of a local
cohomology spectral sequence \cite{ringlct}.

\section{Normalizations and the symmetric Gorenstein context}
\label{sec:normsymm}
In this section we recall from \cite{Dsg} the apparatus for defining
and working with $\CBG$.

\subsection{Normalization}
In commutative algebra, a Noether normalization of $R$ is a regular
subring over which $R$ is finitely generated as a module. There is a
convenient counterpart to this in our context. 

For any finite group $G$ we may choose a faithful 
representation $\rho : G\lra U$ into a connected compact Lie group
$U$ (such as the unitary group $U=U(n)$). This gives a map   $q =\rho^*: \CBU \lra \CBG$ of ring
spectra. 

  Since $U$ is connected $\CBU$ is  regular in the sense that $H_*(\Omega 
  (BU_p))$ is finite dimensional, because   $\Omega (BU_p)\simeq (\Omega BU)_p\simeq U_p$.
   The finite generation statement
  corresponds to the fact that the cofibre $\CBG\tensor_{\CBU} k\simeq C^*(U/G)$ is 
finitely built from $k$.  Accordingly we consider $q$ to be 
 a {\em normalization} of $\CBG$.

\subsection{Finitely generated modules}
For regular local rings, finite generation is equivalent to smallness,
so  we may reasonably say  that $\CBU$-modules are fg precisely if
they are small. 

\begin{defn} \cite{Dsg} A $\CBG$-module $M$ is {\em finitely generated (fg)} relative to $\rho$ if 
$q^*M$ is small over $\CBU$. The {\em bounded derived category}
relative to $\rho$  is the full triangulated subcategory  $\Db (\CBG)$ of fg modules. 
\end{defn}

This is independent of the choice of $U$ and $\rho$ and has the very
concrete characterisation that we used in the introduction. 

\begin{thm}\cite[7.5]{Dsg}
$$\Db (\CBG)=\{ M \st \pi_*M \mbox{ is finitely generated over } \HBG \}$$
\end{thm}

\subsection{The symmetric Gorenstein context}
\label{subsec:symmGor}

The representation $\rho$ induces a fibration 
$$U/G\lra BG\lra BU, $$
which remains a fibration after $p$-completion and hence we have a 
cofibre sequence 
$$\xymatrix{
Q\ar@{=}[d]&R\ar@{=}[d]&S\ar@{=}[d]\\
C^*(U/G) &C^*(BG)\lto &C^*(BU)\lto^{q} 
}$$
of $k$-algebras and an associated cofibre sequence 
$$\xymatrix{
\mcD\ar@{=}[d]&\cE\ar@{=}[d]&\cF \ar@{=}[d]\\
C_*(\Omega((U/G)_p))
\rto^\phi &C_*(\Omega (BG_p))\rto &C_*(\Omega (BU_p)) 
}$$
of $k$-algebras. We note that $C_*U\simeq C_*(\Omega BU_p)$ and we will
abbreviate $\COUG=C_*(\Omega((U/G)_p))$. 

The map $q: S\lra R$  is a normalization since $S=\CBU$ is 
regular (as $U$ is mod $p$ finite) and  $U/G$ is mod $p$-finite.
The map $\phi: \mcD \lra \cE$ is a normalization since $\mcD=\COUG$ is 
regular (as $U/G$ is mod $p$ finite) and  $U$ is mod $p$-finite. 

We say that a $\COBG$-module $X$ is finitely generated (fg) 
relative to $\rho$ if $\phi^*X$ is small over $\COUG$.
These two notions of finite generation correspond under Koszul duality.

\begin{thm}\cite[9.1]{Dsg}
Koszul duality induces  an equivalence
$$\Db (\CBG)\simeq \Db (\COBG)$$
between bounded derived categories of $\CBG$ and 
$\COBG$.
\end{thm}

\section{The Tate localization}
\label{sec:Tate}
In this section we introduce the $\Omega$-Tate spectrum, the 
fundamental object of this paper.  The definition as a Bousfield
localization is formal. Later we will establish good behaviour under  finiteness
conditions.

\subsection{Tate localizations}
When $R=\CBG$ or $\COBG$, the field $k$ is an $R$-module, so we may
consider the Bousfield localization which nullifies the localizing subcategory generated by
$k$. The localization $M\lra L_kM$ is characterised in the homotopy category by the fact that
$[k,L_kM]=0$ and the the mapping fibre is built from $k$.  We will
write
$$\Gamma_k M \lra M \lra L_kM $$
for the associated triangle.  Thus $L_kM$ is the localization of $M$
{\em away from} $k$, and $\Gamma_k M \lra M$ is the
$k$-cellularization of $M$. The functor $L_k$ is monoidal and therefore takes ring spectra to ring
spectra.

One of the powerful features of these localizations is invariance 
under change of base ring. We record it here since we will make 
repeated use of it, both here and in the non-commutative case. 

\begin{lemma}
  \label{lem:cob}
  If $k$ is an $R$-module and we are given a ring map $\phi: S\lra R$, 
  then for any $R$-module $M$, $\phi^*L_k M\simeq 
  L_k\phi^*M$, and similarly for $\Gamma_k$. 
  \end{lemma}

\begin{proof}
We need to observe that $\phi^*M\lra \phi^*L_kM$
has the unversal property of nullification of $k$. The fibre is 
$\phi^*\Gamma_kM$, and is built from $k=\phi^*k$, and 
$$[k, \phi^*L_kM]=[R\tensor_Sk, L_kM]=0; $$
since  $R$ is built from $S$, $R\tensor_Sk$ is built from $k=S\tensor_Sk$. 
\end{proof}

The key definition is simply stated. 
\begin{defn}
\label{defn:OmegaTate}
The {\em Tate localization} of $R$ is the ring  $L_kR$.
  
(i) When $R=\CBG$ we obtain the {\em classical Tate spectrum} 
$L_kR\simeq \hat{C}^*(BG)$, which is the fixed point
spectrum of the usual Tate construction (\cite{Tate, GMTate}). Its
homotopy is the classical Tate cohomology ring.

(ii) When $R=\COBG$ we obtain the {\em $\Omega$-Tate spectrum}
 $L_k\COBG$. Its homotopy is the $\Omega$-Tate homology ring. 

\end{defn}

Altogether, in our case this gives  a web of maps as follows.

$$\xymatrix{
 U/G \rto &  BG \rto & BU\\
&\hat{C}(BG) & \hat{C}^*(BU) \lto \\
C^*(U/G)  &\CBG \uto\lto& \CBU \lto_{q} \uto 
 \\
 \COUG  \rto^{\phi} \dto & \COBG \rto\dto   & \COBU   \\ 
L_k\COUG\rto &   \TCOBG  &  
}$$

We note that since $C^*(U/G)$ and $C_*(U)$ are finitely built from
$k$, they are annihilated by $L_k$. 

\begin{remark}
\label{rem:fgOmegaTate}
The object $k$ is not a module over $\TCOBG$, so there is no obvious
notion of Koszul duality for Tate localizations. Nonetheless, we will say that a
$\TCOBG$-module $X$ is fg if and only if $(L_k\phi)^* X$ is small over
$\TCOUG$, and we write $\Db (\TCOBG)$ for the full subcategory of the
homotopy category consisting of finitely generated modules. 
\end{remark}

  \section{Anderson-Tate duality for $L_k\COBG$}
\label{sec:AndOmega}
We show that the  same condition that gave 
control over the Tate localization gives 
a duality statement for $\COBG$ of the form familiar from 
the duality on Tate cohomology of finite groups. 

In this section  $G$ is a compact Lie group of dimension $d$.
As in Subsection \ref{subsec:GorD}
said it is  {\em orientable over $k$} if the action of the component
grou on $H_d(G)$ is trivial. The reader may restrict attention to the
case of a finite group, but the role of the dimension in the more
general case is illuminating.

\subsection{Classical Tate duality for finite groups}
\label{subsec:classicalTate}
  We restate the classical situation in our language as a template.  For a finite group $G$ we have the norm sequence
  $$C_*(BG)\stackrel{\nu}\lra \CBG\lra \ChatBG.$$
  The only degree in which $\nu$ may be non-zero is degree zero; since
  $G$ acts trivially on coefficient group $k$,  the norm is multiplication by the group
  order. If  $p$ does not divide the group order then it is an
  isomorphism and $1=0$ in Tate cohomology so 
  the Tate cohomology is zero. If $p$ does
  divide the group order $\nu_*=0$ and  we have   a short exact
  sequence
  $$0\lra H^*(BG)\lra \hat{H}^*(BG)\lra \Sigma H_*(BG)\lra 0.$$
Since homology and cohomology are dual, we obtain the Tate duality
statement
that the positive codegrees are dual to
  negative degrees with a shift
  $$\hat{H}^n(BG)=H^n(BG)=\Hom(H_{-n}(BG),
  k)=\Hom(\hat{H}_{-1-n}(BG),k), $$
  or
  $$\hat{H}^*(BG)\simeq \Sigma^{1}(\hat{H}^*(BG))^{\vee}.$$

  Similarly for a compact Lie group $G$ of dimension $d$ with $k$-orientable adjoint
  representation, where $C_*(BG)$ is replaced by $\Sigma^dC_*(BG)$; 
  if $d$ is positive $\nu_*$ is automatically zero and 
  $$\hat{H}^*(BG)\simeq \Sigma^{d+1}(\hat{H}^*(BG))^{\vee}.$$

  \subsection{Gorenstein duality for $\COBG$}
We follow the template of Subsection \ref{subsec:classicalTate} for
the Koszul dual ring. 
  By \cite[10.3]{DGI}, if the adjoint representation of $G$ is
$k$-orientable then   $ \CBG$ is Gorenstein of shift $d$ and has
Gorenstein duality of shift $d$, so that $\Gamma_k\CBG\simeq \Sigma^d C_*(BG)$.

The argument of \cite[8.5]{DGI} shows that $\COBG$ is also Gorenstein of
shift $d$, but to make Gorenstein {\em duality} statements we need
to calculate the homotopy of $\cell_k \COBG$.

  \begin{thm}
\label{thm:GorDCOBG}
    If $G$ the adjoint representation is $k$-orientable, then
    $$\cell_k\COBG \simeq \Sigma^d\CCOBG.$$
  \end{thm}

  \begin{proof}
  Since $R=\COBG$ is a Gorenstein $k$-algebra of shift $d$, we have
  $$\Hom_R(k,R)\simeq \Sigma^dk\simeq \Hom_R(k, \Sigma^d
  \Hom_k(R,k)). $$
   If $R'=\CBG$, the action on $\Hom_{R'}(k,R')$ is the
  action of $\pi_1(BG_p)$ on   $H^d(S^{ad(G)};k)$, which is trivial by
  hypothesis, and by the Morita invariance of the Gorenstein condition
  \cite[8.5]{DGI}, the action of $R=\COBG$ is the
  corresponding action.  Hence, the above composite isomorphism
  respects right $\cE$-actions, and by effective constructibility (Subsection \ref{subsec:effcon}), we 
conclude 
$$\cell_kR\simeq \Sigma^d\cell_k(R^{\vee})\simeq \Sigma^d R^{\vee}.$$
   \end{proof}

\subsection{Anderson-Tate duality for the Koszul dual}  
  In terms of ring spectra, if $R$ is an augmented $k$-algebra, we are
  taking the cofibre sequence
  $$\Gamma_kR\lra R\lra L_k R. $$
  If $R=\CBG$  and the adjoint representation is $k$-orientable then
  by Gorenstein duality,   the cellularization takes the expected form
  $$\Gamma_k \CBG \simeq \Sigma^{d}C_*(BG), $$
  where $d$ is the dimension of $G$. This recovers the discussion in
  Subsection \ref{subsec:classicalTate}.

  However, the ring  $R=\COBG$ again has Gorenstein duality of shift $d$,
  and hence
  $$\Gamma_k \COBG \simeq \Sigma^{d}\CCOBG. $$
  Again this gives a cofibre sequence
  $$\Sigma^d\CCOBG\stackrel{\nu}\lra \COBG\lra L_k\COBG, $$
  but now the suspension means that $\nu_*$ is potentially non-zero in
  degrees betwee $0$ and $d$. Furthermore the special case where $G=U$
  is a compact connected Lie group has $\Omega BU\simeq U$ finite and
  hence $L_kC_*(U)\simeq 0$ so that $\nu$ is an equivalence. We will
  study $\nu$ in more detail in Section \ref{sec:nu}.

  \begin{thm}
\label{thm:Lknorm}
Provided  the adjoint representation of $G$ is $k$-orientable
then
$$\pi_nL_k\COBG=\dichotomy
{H_n(\OBG) \mbox{ if } n\geq d+2}
{H^{d+1-n}(\OBG) \mbox{ if }  n\leq -1.}$$
In the remaining degrees there is an exact sequence 
$$H^{d-n}(\OBG)\lra H_n(\OBG)\lra   \pi_nL_k\COBG\lra 
H^{d+1-n}(\OBG)\lra H_{n-1}(\OBG), $$

The group $G$ is a $p$-compact group if and only if 
 $L_k\COBG\simeq 0$.
If $G$ is finite, it is $p$-nilpotent  if and only if 
 $L_k\COBG\simeq 0$.  If $G$ is a finite group and
not $p$-nilpotent then 
$$\pi_0(L_k\COBG)=k\Pi \oplus  H^1(\OBG) \mbox{ and }
\pi_1(L_k\COBG)=k\Pi \oplus  H_1(\OBG) , $$
where  $\Pi =G/O^p(G)$ is the largest $p$-quotient of $G$. 
\end{thm}

\begin{proof}
  After Theorem \ref{thm:GorDCOBG} the only things still requiring proofs concern when the norm map is  trivial.   It is proved by Bousfield and Kan that if $G$ is finite 
$\pi_1(BG_p^{\wedge})=G/O^p(G)$.

  It is clear that if $G$ is $p$-nilpotent 
  then $\COBG=C_*(G/O^p(G))$ is finite dimensional. The definition of
  $p$-compact is that $\COBG$ is finite dimensional. In any case, this
  means $L_k\COBG\simeq 0$.
  Conversely, if $L_k\COBG\simeq 0$ then $\COBG$ is $k$-cellular. By
  effective constructibility we have the equivalence
  $$\Sigma^d\CCOBG\simeq \Sigma^dk\tensor_{\CBG}k\simeq
  \Hom_{\COBG}(k, \COBG)\tensor_{\CBG}k \simeq \COBG. $$
In particular this implies $\HOBG$ is in a finite range, and of finite
total dimension, so $G$ is a $p$-compact group.

Otherwise if $L_k\COBG\not \simeq 0$ the norm map is not an
isomorphism.

Finally, we we must show that if $G$ is not $p$-nilpotent then the norm map
is zero.  We defer this to Section \ref{sec:nu}, and give an
alternative approach using Benson's squeezed
resolutions in Section \ref{sec:squeeze}. 
\end{proof}

\section{The norm}
\label{sec:nu}
In this section we continue with the cofibre sequence
$$\Gamma_k R\stackrel{\nu}\lra R\lra L_kR$$
in various cases, with  $R=C_*(G)$ or 
$\COBG$ with a view to understanding $\nu$.
Our principal concern  is for  a finite group $G$, where we give a
complete analysis, but we will also
consider more general compact Lie groups at the end.

The idea is that it is elementary to understand the map $\nu$ when
$R=C_*(G)$, and we can deduce substantial information from the
double centralizer map $C_*(G)\lra \COBG$.

\subsection{Context}
We suppose $G$ is a compact Lie group of dimension $d$ and the adjoint
representation is orientable over $k$.
The rings$R=C_*(BG), \CBG$ or  $\COBG$, are then Gorenstein of shift $d$, with 
$$\Hom_R(k, R)\simeq \Sigma^d k . $$
It is elementary that $\Hom_R(k, R^{\vee})\simeq k$. 
For $R=\CBG$ or $\COBG$ we have seen $R^{\vee}$ is $k$-cellular and 
$\Gamma_k R\simeq \Sigma^d R^{\vee}$. We will discuss the more elementary case of
$R=C_*(G)$ below. By definition the map 
$$ \Gamma_k R 
\stackrel{\nu}\lra R$$
is the one inducing a $k$-equivalence. We may understand the map $\nu$
since 
$$\Hom_R(\Gamma_kR , R) \simeq R_k^{\wedge}$$
If $R=\CBG $ or $\COBG$ then 
$ R_k^{\wedge}\simeq R $, so that $R_0=k$
and we find $\nu$ is determined by a scalar. 
It remains to identify the map 
$(\Sigma^d R_*)^{\vee}=[R, \Gamma_kR]\lra [R, R]=R_*.$

\subsection{The group ring of a finite group}
First we take $G$ to be finite and consider $R=C_*(G)=kG$. As before
$O^p(G)$ is the smallest normal subgroup so that $G/O^p(G)=\Pi$ is a
$p$-group. The group $G$ is said to be $p$-nilpotent if the Sylow
$p$-subgroup is a retract of $G$, or equivalently if $O^p(G)$ is a
$p'$-group.

At the level of abelian categories, we may consider modules $T$ which are $k$-cellular in the sense
that they have filtration with subquotients a sum of copies of the
trivial module $k$. The $k$-cellular obstructor of $M$ is the smallest
submodule of $M$ so that $M$ is $k$-cellular, and Benson observes
\cite[3.1]{BensonSqueezed} that the $k$-cellular obstructor is
$[O^p(G),M]=\{ m-gm \st m\in M, g\in O^p(G)\}$. The map
$M\lra M/[O^p(G), M]$ is thus an initial map to a $k$-cellular
module. Similarly the fixed point submodule $M^{O^p(G)}$ (which can be
viewed as a $k\Pi$-module) is the largest $k$-cellular submodule of
$M$, and the inclusion $M^{O^p(G)} \lra M$ is the terminal map from a
$k$-cellular submodule.

The counterpart of $\Gamma_k M\lra M \lra \Lambda_k M$ is therefore
$M^{O^p(G)}\lra M \lra M/[O^p(G),M]. $

In the following discussion, for a set $X$ of 
group elements we write $[X]=\Sigma_{x\in X}x$.

\begin{lemma}
  \label{lem:Opzero}
  The composite $(kG)^{O^p(G)}\lra kG\lra kG/[O^p(G), kG]=k\Pi$
  is an isomorphism if $G$ is $p$-nilpotent and zero otherwise.
\end{lemma}

\begin{proof} The module $\Gamma_k kG=kG^{O^p(G)}$ is spanned by coset sums. 
Since the action of $O^p(G)$ is trivialized in the codomain, the coset
sum $[xO^p(G)]$ is represented by  $|O^p(G)|x$.
The group $G$ is $p$-nilpotent if and only if  $O^p(G)$ is a $p'$-group. 
\end{proof}

\subsection{The $\Omega$-homology of a finite group}
Continuing with a  finite group, we consider  $R=\COBG$.
The idea is to deduce the answer from $R=C_*(G)$ using the fact that we
have maps
$$C_*(G)\lra \COBG\lra H_0(\Omega BG_p),  $$
where the first is the double centralizer map, and the second kills
homotopy in degrees 1 and above.  Since $G$ is finite this is
$$kG\lra \COBG \lra k\Pi. $$

\begin{lemma}
  \label{lem:nuzero}
 If $G$ is finite and  $R=\COBG $, the map $\nu_*$ is an isomorphism if 
 $G$ is $p$-nilpotent and otherwise is the zero map. 
 \end{lemma}

 \begin{proof} We only need to consider degree 0. 
The maps $kG\lra k\Pi =\COBG\lra k\Pi$,
show $\COBG$ is augmented over $k\Pi$, and in particular we have a
commutative square
$$\xymatrix{
  (k\Pi)^{\vee} \dto\rto^{\nu}&k\Pi\dto\\
  (\COBG)^{\vee} \rto^{\nu}&\COBG\\
}$$
Since the verticals are isomorphisms in degree 0, the result follows
from Lemma \ref{lem:Opzero}.
\end{proof}

\subsection{Chains on a connected compact Lie group}
We next consider the special case of a connected compact Lie groups
$G$. In this case $\Omega (BG_p)\simeq G_p$ and the double centralizer
map is an equivalence $C_*(G)\stackrel{\simeq} \lra \COBG$

\begin{lemma}
  If  $R=C_*(G)$ with $G$ connected, then  the map 
  $$\nu : \Sigma^d C^*(G_e)\lra C_*(G_e)$$
  is an isomorphism. 
\end{lemma}

\begin{remark}
It is natural to wonder if this is the Poincar\'e Duality
isomorphism. However our map is a map of $H_*(G)$-modules (which
makes sense because $G$ is a group), whereas the
Poincar\'e Duality isomorphism is a map of $H^*(G)$-modules (which
makes sense because $G$ is a space). The rings need not be isomorphic
(for example if $H_*(G)$ is not commutative). 
 \end{remark}

  \begin{proof}
Since $H_*(G)$ is finite dimensional and connected,  $L_kC_*(G)\simeq 0$.  
    \end{proof}

\subsection{Nilpotent action}
 Finally, we make a statement covering both the case of
 a finite group and the case of a connected group. It does not cover
all compact Lie groups, but it does suggest what the simplest behaviour looks
like.

 For a general compact Lie group $G$ we have a short exact sequence
 $$1\lra G_e\lra G\lra G_f\lra 1$$
 where $G_e$ is the identity component of $G$ and $G_f=G/G_e=\pi_0(G)$ is
 the finite quotient. This induces $k$-algebra maps
$$C_*(G_e)\lra C_*(G)\lra C_*(G_f)=kG_f.   $$

In general, if $G_f$ is  a $p'$-group  acting on $G_e$, we see
$H^*(BG)=H^*(BG_e)^{G_f}$. In this case $BG$ and $BG_e$ may be very
different: for example $H^*(BG_e)$ might be polynomial and $H^*(BG)$
might not even be a complete intersection. To avoid these
complexities, we  may suppose $G_f$ acts nilpotently on
$H^*(BG_e)$. In this case, $p$-completion preserves
the fibration and we obtain a $p$-adic fibration
$$ (BG_e)_p^{\wedge}\lra BG_p^{\wedge}\lra (BG_f)_p^{\wedge}. $$
This in turn gives a fibration
$$ (G_e)_p^{\wedge}\lra\Omega( BG_p^{\wedge})\lra \Omega
((BG_f)_p^{\wedge}). $$
In this case, the base is covered by our analysis of finite groups and
the fibre is covered by the case of connected compact Lie groups, and
we may hope  that the behaviour is a simple combination of the two.

\begin{lemma}
Suppose $G_f$ acts nilpotently on $H^*(BG_e)$.

(i) For $R=C_*(G)$, the map $\nu_*$ followed by $H_*(G)\lra \Lambda_k
H_*(G)$ is an isomorphism if $G_f$ is 
$p$-nilpotent and zero if $G_f$ is not $p$-nilpotent.

(ii) For $R=\COBG $, the map $\nu_*$ is an isomorphism if $G_f$ is 
$p$-nilpotent. If $G_f$ is not $p$-nilpotent, $\nu_*$ is  zero on the
submodule of $\Sigma^dH^*(\Omega BG_p)$
generated by the image of $H_*(G)^{O^p(G_f)}\cong H_*(G_e)[\Pi]$.
The image of $\nu_*$ maps to zero in $H_*(\Omega (BG_f)_p)$.    
\end{lemma}

\begin{remark}
If the spectral sequence collapses to an isomorphism 
$$\HOBG=H_*(G_e)\tensor H_*(\Omega (BG_f)_p)$$
and $G_f$ is not $p$-nilpotent,  then $\nu_*$ is zero. 
\end{remark}

    \begin{proof}
We know that $H_0(G)=kG_f$ and $H_i(G)=H_i(G_e)[G_f]$ (the action of
$G_f$ on $H_i(G_e)$ untwists in this induced module). The argument of
Lemma \ref{lem:Opzero} shows that the map 
$$\nu_*: H_*(G_e)[\Pi]=H_*(G_e)[G_f]^{O^p(G_f)}\lra H_*(G) \lra
H_*(G)/[O^p(G_f), H_*(G)]=H_*(G_e)[\Pi]$$
is multiplication by  $|O^p(G_f)|$, and it follows that the map is an
isomorphism if $G$ is $p$-nilpotent and 
zero if $G_f$ is not $p$-nilpotent.

For $R=\COBG$,  we see that $\nu$ is an equivalence if $G_f$ is
$p$-nilpotent (since $\HOBG$ is finite dimensional and built from the
trivial representation).  If $G$ is not $p$-nilpotent we deduce the
the statement from the comparison maps $C_*(G)\lra \COBG \lra
C_*(\Omega (BG_f)_p)$.
    \end{proof}


  \section{The squeezed Tate construction}
  \label{sec:squeeze}
In this section $G$ is a finite group, and we recall Benson's purely
algebraic approach to $\HOBG$, extending it slightly to include the
squeezed Tate construction. 
\subsection{Squeezed resolutions}
Benson \cite{BensonSqueezed} has shown how to calculate the homology
and cohomology of $\OBG$ in purely algebraic terms. We explain how to
incorporate the norm map $H^0(\OBG)\lra H_0(\OBG)$ into this framework
and then calculate.

An ordinary projective resolution of a module $A$ can be described as repeatedly
taking the kernel of a projective cover: we choose a surjective map
$P(A)\lra A$ from a projective module $P(A)$ and take the kernel
$$0\lra \Omega A \lra P(A)\lra A\lra 0. $$
The resolution is obtained by iterating the process: we take $A_0=A$
and then, once $A_i$ is defined, we take $P_i=P(A_i)$
and
$$A_{i+1}=\Omega A_i=\ker(P_i\lra A_i).$$

For the squeezed resolution there is another step. 
For a module $B$ the squeezed submodule $B_{\sigma}$  is the
smallest submodule so that $B/B_{\sigma}$ is built from the trivial 
module $k$ (we have $B_{\sigma}=[O^p(G), B]=\{ x-\gamma x\st \gamma \in 
O^p(G), x\in B\}$). The squeezed loops on $A$ is defined to be
$$\Omega_{\sigma}A=(\Omega A)_{\sigma}.$$
We now define a {\em left squeezed resolution} of a $kG$-module 
$A$ by taking $A_0=A$ and, once
$A_i$ is defined,  we take $P_i=P(A_i)$ and 
$A_{i+1}=\Omega_{\sigma}A_i$. In more detail, we define 
$P_i, B_i, A_{i+1}$ as follows 
$$A_i\lla P(A_i)=P_{i}\lla \Omega A_i=B_{i}\supseteq (B_{i})_{\sigma}=A_{i+1}. $$
This gives  a  {\em left squeezed resolution} 
$$\cdots \lra P_2\lra P_1\lra P_0\lra 0$$
and we  then define $\HO_*(G;A)=H_*(P_{\bullet})$. 

\begin{remark}
  If $A=k$ the left squeezed resolution coincides with that in
  \cite{BensonSqueezed}. We take $A_0=k$ and then
  $$A_0P_0B_0A_1P_1B_1\ldots$$
  whereas, in corresponding steps, Benson takes
  $$M_{-1} P_0N_0M_0P_1N_1\ldots $$
  Actually, Benson does not define
  $M_{-1}$, and begins by taking $P_0=N_0=P(k)$ and then takes
  $M_i=(N_i)_\sigma$, $P_{i+1}=P(M_i)$, $N_{i+1}=\Omega M_i$. 
 It is easy to see that the projectives $P_i$  in the two sequences agree, and
 corresponding terms in the two sequences agree from $A_1=M_0$ onwards, since
  $$P(k)_\sigma =(\Omega k)_\sigma. $$
\end{remark}

Similarly, if $D$ is a $kG$-module, we define $D^{\sigma}$ to be the
quotient of $D$ by the largest submodule built from $k$, and we take
$\Omega^{-1}_\sigma A=(\Omega^{-1}A)^\sigma$.
We then
define a {\em right  squeezed resolution} of a $kG$-module 
$C$ by taking $C_0=C$ and, once 
$C_i$ is defined,  we take  $I_i=I(C_i)$ (injective hull) and 
$C_{i-1}=\Omega_{\sigma}^{-1} C_i$. In more detail, we define 
$I_i, D_i, C_{i-1}$ as follows 
$$C_i\lra I(C_i)\lra \Omega^{-1}C_i=D_i\lra D_i^{\sigma}=C_{i-1}$$
This gives  a  {\em right squeezed resolution} 
$$0 \lra I_0\lra I_{-1}\lra I_{-2}\lra\cdots $$
 and then define $\HOC^*(G; C)=H^*(I_{\bullet})$.
 Once again, for analagous reasons to the left case,
 when $C=k$ the complex $I_{\bullet}$
 coincides with Benson's right squeezed resolution.

Benson has shown \cite[Theorem 1.2]{BensonSqueezed} that
$\HO_*(G;k)=\HOBG$ and $\HOC^*(G;k)=H^*(\OBG)$.

\subsection{The Tate squeezed resolution}

For the module $k$, we  splice the left and right squeezed resolutions
together to form a Tate resolution $T_{\bullet}$ by using a cofibre sequence 
$$I_{\bullet}\stackrel{n}\lra P_{\bullet}\lra T_{\bullet}. $$
Such a map $n$ is specified by an arbitrary map $n_0: I_0\lra P_0$,
and we take $I_0=P_0$ and $n_0$ to be the identity.
We then define the squeezed Tate homology by $\HOT_*(G)=H_*(T_{\bullet})$.

\begin{lemma} (Benson) 
  \label{lem:Benson}
Taking $n_0$ to be the identity gives the norm map
  $$n_*: \HOC^0(kG)\lra \HO_0(kG)$$
  given  by multiplication by $|O^pG|$ and is therefore an 
  isomorphism if $G$ is $p$-nilpotent and 0 if $G$ is not $p$-nilpotent. 
\end{lemma}
\begin{proof}
 By the comparison map \cite[3.4]{BensonSqueezed} we may take
 $I_0=P_0=kG$.  There is a commutative square 
  $$\xymatrix{
    kG\rto^{id}  & kG\dto \\
    kG^{O^p(G)}\uto&kG/[O^p(G),kG]\\
    \HOC^0(G)\ar@{=}[u]\rto^{n_*} & \HO_0 (G)\ar@{=}[u] 
  }$$
The statements about the maps were proved in Lemma \ref{lem:Opzero}. 
  \end{proof}

To tie this algebraic treatment to the homotopical treatment, we need
to show $n$ is $k$-cellularization. Benson gives a 
$H^{\Omega}_*(G)$-module structure by comparison of resolutions
\cite[4.6]{BensonSqueezed}. Because the arguments are formal
consequences of connectivity this is sufficient to establish the
result required. 

\begin{lemma}
The map $P_{\bullet}\lra T_{\bullet}$ is nullification of $k$, and
therefore the squeezed Tate homology is the $\Omega$-Tate homology:
$$\HOT_*(G)=\pi_*(L_k\COBG). $$
\end{lemma}

\begin{proof}
  To see the map is a $k$-equivalence we note that by the squeezing
  process
  $\Hom_{kG}(k, I_{\bullet})=I_0^G$ and  $\Hom_{kG}(k,
  P_{\bullet})=P_0^G$. To see that $I_{\bullet}$ is $k$-cellular, we
  need only note that since $H_*(I_{\bullet})$ is bounded above the
  process of killing homotopy groups shows
  $I_{\bullet}$  is the homotopy direct limit of bounded below
  complexes with $k$-cellular homology.
  \end{proof}

\section{The singularity category}
\label{sec:sing}
Now that we have all the ingredients, we recall the definition of the
singularity category of $\CBG$.

\subsection{The definition}
For a commutative local ring $R$ the singularity category is the
Verdier quotient
$$\Dsg (R):=\frac{\Db (R)}{\sfD^c (R)}$$
where the numerator (the bounded derived category) may be defined as 
the homotopy category of finite complexes of finitely generated 
modules, and the denominator (the derived category of compact objects) 
may be seen to be the homotopy category of finite complexes of 
finitely generated projectives. The point of the definition is that 
(by the Auslander-Buchsbaum-Serre Theorem) 
$\Dsg (R)$ is trivial if and only if the local ring $R$ is  regular. Its 
nontriviality therefore measures the deviation of $R$ from being 
regular. It has especially good formal properties if $R$ is 
Gorenstein.

We need to extend the definition of the singularity category to 
$R=\CBG$ in a way that respects Koszul duality. We described the 
symmetric Gorenstein context of \cite{Dsg} in Subsection \ref{subsec:symmGor}. 
\subsection{BGG correspondence}
\label{subsec:BGG}
We explained the definition of the bounded derived category in Section
\ref{sec:normsymm} we
may define the singularity category
$$\Dsg (R) :=\frac{\Db (R)}{\sfD^c (R)}   $$
when $R=\CBG$ or $\COBG$.

Alongside this,  we have the
definition of the dual {\em cosingularity} category
$$\Dcsg (R) :=\frac{\Db (R)}{\langle k \rangle}, $$
where the denominator is the thick subcategory generated by $k$.
Serre's Theorem establishes the importance of the invariant:  for a graded
connected $k$-algebra  $R$, it describes
quasicoherent sheaves over
$\mathrm{Proj}(R)$: $\Dcsg (R)\simeq \Db (\mathrm{Proj}(R))$.
We may thus think of the cosingularity category as very geometric in
flavour. 

The Morita equivalence of bounded derived categories exchanges compact
objects (finitely built from the ring) and objects finitely built from $k$.

\begin{thm}\cite[9.10]{Dsg}
\label{thm:BGG}
 Koszul duality induces the equivalences
$$\Dsg (\CBG)\simeq \Dcsg (\COBG) \mbox{ and }\Dsg (\COBG)\simeq \Dcsg
(\CBG). $$
\end{thm}

Note that it follows in particular that $\Dsg(\CBG)$ is trivial if and only if
$\CBG$ is regular. 

To see the relationship to  the BGG correspondence consider the 
second  when $G$ is an elementary abelian $2$-group of rank $r$: in
the light of Serre's Theorem it  shows the singularity category of an exterior algebra on
$r$-generators is the derived category of quasi-coherent sheaves on 
the projective space $\mathbb{P}^{r-1}$ in the familiar way.

  \section{Modules over the Tate localization}
  \label{sec:modovertate}
In this section we show that the singularity category of $\CBG$ is the
bounded derived category of the $\Omega$-Tate spectrum. The result
involves understanding the interactions of finiteness conditions, and
it is remarkable that the result holds for all finite groups $G$, and
all compact Lie groups for which the adjoint representation is
$k$-orientable. 

In the following proofs we have in mind a normalization 
$$U/G\lra BG\lra BU, $$
or rather its $p$-completion. The space $(U/G)_p$ is finite, the space
$BU_p$ has a finite loop space, and $BG_p$ itself is acquires finiteness
properties from this fibration.

\subsection{Truncations}

We first note that for the relevant spaces $X$, the homotopy $\pi_*(L_k\COX)$ is finite dimensional in each
degree and agrees with $\HOX$ above some degree. 

\begin{lemma}
\label{lem:degreewisefinite}
If $X$ is the $p$-completion of $U/G, BG$ or $BU$ then 
the homotopy groups $\pi_s( L_k\COX)$ are finite dimensional for each $s$
and $\pi_*(L_k\COX)$ coincides with  $\HOX$ above some degree. 
\end{lemma}

\begin{proof}
Since $\Omega BU_p $ is finite, this is obvious for $BU$. We have
shown this holds for $BG$ in Theorem \ref{thm:Lknorm} above, 
and it follows for $U/G$ from the other terms. 
\end{proof}

\begin{lemma}
  \label{lem:smalltruncationsofTCOUG}
  If $F$ is finite  then  for every $s\in \Z$ the truncation
  $\taus \TCOF$ of  $\TCOF$ is     small over $\COF$. 
\end{lemma}

\begin{proof}
By Lemma \ref{lem:degreewisefinite}, the truncation only differs
from $\TCOF$ in finitely many degrees. Since  each homotopy group of $\TCOF$ is finite 
dimensional over $k$ the overall difference is finitely built from
$k$.  Finally, since $F$ is
finite, $k$ is small over $\COF$ (indeed, $k$ finitely builds $C^*(F)$
and hence $\COF$ finitely builds $k$) and the truncation differs from $\TCOF$ by a small module. 
\end{proof}

\begin{lemma}
  \label{lem:smalltruncations}
      If  $F=U/G_p$ and $M$ is a small
      $\TCOF$-module then  for any $s\in \Z$ the  truncation $\taus M$
      is a small $\COF$-module. 
    \end{lemma}

    \begin{proof}
Every small object is finitely built from $\TCOF$, so we may prove 
the result by induction on the number of steps required. Lemma 
\ref{lem:smalltruncationsofTCOUG} gives the base of the induction.

We need to show that the property is 
preserved by  adding a single cell (since it is obviously preserved by 
passage to retracts). 

Suppose then that all truncations of $M$ are small, and that we have a 
cofibre sequence 
$$\Sigma^i\TCOF \lra M \lra M'.$$
The map $M\lra M'\lra M'(-\infty , -1]$ factors through $M\lra 
M(-\infty , -1]$ so we may form the diagram
$$\xymatrix{
  Z_1\rto \dto & \tauz M\rto \dto &\tauz M'\dto\\
  \Sigma^i \TCOF \rto \dto & M \rto \dto & M'\dto\\
  Z_2\rto & M(-\infty , -1]\rto & M'(-\infty , -1]. 
}$$
From the bottom row, the spectrum $Z_2$ has no homotopy in degrees
$\geq 0$. From the top row,  $Z_1$ only 
has homotopy in degrees $\geq -1$. The left hand vertical then shows
that  $Z_1$  differs from 
the truncation of $\Sigma^i \TCOF$ in a finite dimensional vector 
space.  By Lemma \ref{lem:degreewisefinite},  $\tauz \Sigma^i \TCOF$
differs from $\Sigma^i \COF$ in a finite dimensional vector space. Since $k$ is small over  
$\COF$ (as in Lemma \ref{lem:smalltruncationsofTCOUG}), it follows that $Z_1$ is small. From the top row we see 
that since $Z_1$ and $\tauz M$ are small, so is $\tauz M'$. 
\end{proof}

We apply the previous lemma with $F=U/G_p$ to the normalization
$\phi : \COUG\lra \COBG$. 

\begin{cor}
  \label{cor:fgtruncations}
If $N$ is a fg $\TCOBG$-module then every truncation of $N$ is a fg
$\COBG$-module. 
  \end{cor}

\begin{proof}
By hypothesis $M=\phi^*N$ is small over $L_k\COUG$. By Lemma
\ref{lem:smalltruncations}, 
all truncations $\taus M$ are small over $\COF$. However 
$\taus M=\taus \phi^*N=\phi^*\taus N$, so $\taus N$ is a fg $\COBG$-module as
required. 
\end{proof}

\subsection{The singularity category as a bounded derived category}
We are now equipped to outline the strategy for understanding $\Dsg
(\CBG)$. Theorem \ref{thm:BGG} shows $\Dsg (\CBG)\simeq \Dcsg
(\COBG)$, and we continue by showing this is the bounded derived
category of the $\Omega$-Tate spectrum $\TCOBG$.

\begin{thm}
  \label{thm:csgisbLk}
Extension of scalars along $\COBG\lra \TCOBG$ induces an equivalence 
$$\Dcsg (\COBG)\simeq \Db ( \TCOBG). $$
\end{thm}

  \begin{proof}
    At the level of unrestricted module categories,
    extension of scalars induces a functor $\COBG\modules\lra
\TCOBG\modules$. To see this induces a map on bounded derived
categories we need to show that an fg-module over $\COBG$ maps to an
fg-module over $\TCOBG$.

Thus we suppose given a fg $\COBG$-module $Z$, meaning that $\phi^*Z$ is
small over $\COUG$.

Since $L_k$ is smashing, the image of $Z$ in $\TCOBG\modules$ is
$L_kZ$, and $\phi^*L_kZ=L_k \phi^*Z$. Hence if $Z$ if fg we conclude
$L_kZ$ is fg as required.

This gives a map 
$$\nu: \Db  (\COBG)\lra \Db ( \TCOBG). $$
It is clear that $\nu (k)\simeq 0$ and since $L_k$ is exact, $\nu$ induces
$$\nubar: \Dcsg  (\COBG)=\Db  (\COBG)/\langle k \rangle
\lra \Db ( \TCOBG). $$
We must show this is full, faithful and essentially surjective. We
will first consider the map on objects.

{\bf Injectivity:}
To see that $\nubar$ is injective we note that the kernel of $L_k$ is
precisely the localizing subcategory generated by $k$. It remains to
say that if  $M$ is a $\COBG$-module with $M\simeq \Gamma_kM$ and it 
is small over $\COUG$ then $M$ is finitely built from $k$. Since $M$ is a retract of
$\COBG\tensor_{\COUG}M$ and $\COBG\tensor_{\COUG}k \simeq C_*(U)$ is
finitely built from $k$, it suffices to check that $M$ is finitely
built from $k$ as a $\COUG$-module. 

By hypothesis $M$ is finitely built from $\COUG$, and hence $\Gamma_k
M$ is finitely built from $\Gamma_k\COUG$. Since $\Gamma_k\COUG\simeq
C^*(\Omega U/G_p)$, and $H^*(\Omega U/G_p)$ is finite dimensional in each
degree, we see $\Gamma_k \COUG\simeq \colim_n \Gamma_kM^{(n)}$
with each term finitely built from $k$. 
Accordingly, the identity factors through  $\Gamma_kM^{(n)}$
for some $n$, and hence $\Gamma_k M$ is a retract of an object
finitely built from $k$. 

{\bf Surjectivity:} We suppose $N $ is a fg $\TCOBG$-module, so that 
$\phi^*N$ small over $\TCOUG$. We let $ Z=\tauz N$, and note that
$L_kZ\simeq N$ because $L_k$ annihilates any object bounded above. 

We claim that $Z$ is in fact finitely generated. Indeed,
$$L_k\phi^*Z=L_k\tauz \phi^*N\simeq L_k\phi^*\tauz N\simeq
L_k\phi^*N\simeq \phi^*N, $$
so the  localization of $\phi^*Z$ is small. Since $Z$ is the
truncation of $N$, the result follows from Corollary \ref{cor:fgtruncations}.
   
{\bf Morphisms:} 
Since $L_k$ is smashing, morphisms in the the localization of 
$\COBG$-modules are $L_k\COBG$-maps, and by definition the 
localization $L_k$ on $\COBG$-modules is the Verdier 
quotient by the localizing subcategory $\Loc(k)$, so 
$$\sfD(\COBG)/\Loc(k)\simeq \sfD(L_k\COBG).$$
Now consider the diagram 
$$\xymatrix{
\Db (\COBG)\rto \dto^{ff} &\Db (\COBG)/\langle k \rangle\dto \rto &\Db(L_k\COBG)\dto^{ff}\\
\sfD (\COBG)\rto &\sfD(\COBG)/\Loc(k) \rto^{\simeq} &\sfD (L_k\COBG)
}$$
with the outer verticals full and faithful by definition.  
We have argued that the part of $\Loc (k)$ in the 
bounded derived category is the thick subcategory $\langle k \rangle$,
so the central vertical is also full and faithful.

 \end{proof}

Putting results together we obtain a description of the singularity
category as a bounded derived category.

\begin{cor}
  \label{cor:DsgisDb}
We have  equivalences  
$$\Dsg (\CBG)\simeq \Dcsg (\COBG)\simeq \Db ( \TCOBG). $$
\end{cor}
\begin{proof}
  We combine Theorem \ref{thm:BGG}, and Theorem \ref{thm:csgisbLk}. 
  \end{proof}

\section{Central reduction}
\label{sec:centralred}
In the commutative setting (for example for modules over $\CBG$)
there is a well known construction of the
cellularization $\Gamma_k$ and the nullification $L_k$. The purpose of
this section is to give conditions under which these same
constructions work for the noncommutative ring $\COBG$.  The flavour
of the conditions is that $C^*(BG)$ behaves like a complete
intersection. The first requirement is that $\HOBG$ has polynomial
growth. For commutative local rings,  a growth condition is enough to
characterize complete intersections, but for $C^*(BG)$ we need a
little more. The actual condition is about Hochschild cohomology, but
it suffices to assume that $G$ has a normalization in which the
cohomology of the fibre is a complete intersection. This applies for example to  
Chevalley groups at good primes.

\subsection{Polynomial growth}  
To make this useful we need to understand the ring spectrum
$\TCOBG$. At the crudest level we want to understand its coefficient
ring, but we are working towards an understanding of its module
category. 

When $\HOBG$ is periodic with periodicity element $\tau$ (as in the
case of cyclic Sylow subgroup \cite{Dsgpq}), then
 we need only check that $\tau $ may be  taken central, and then as a
 module, the  Tate localization is just a mapping telescope, so that
$\TCOBG=\COBG[1/\tau]$. Furthermore we understand the terms in the
telescope, and the homotopy groups are clear $\pi_*(\COBG[1/\tau])=(\HOBG)[1/\tau]$. 

The general situation is more complicated. It is familiar
from commutative algebra that the singularity category behaves much
better for complete intersections.  In fact we can
make good progress here under a finiteness assumption directly analagous to
the growth condition that characterises complete intersections in
commutative algebra. 

We will use some results from the study of complete intersections
\cite{pzci}, starting with the growth condition.

\begin{defn} \cite{pzci}
A $p$-complete space $X$ is said to be {\em gci}  if $H_*(\Omega X)$ has
polynomial growth (the letter $g$ refers to the fact that this is a
growth condition). 
\end{defn}

It is essentially due to the work of Felix-Halperin-Thomas \cite{FHT} that this
finiteness condition gives good control over the structure of the homology.

\begin{lemma}
  \label{lem:finiteoverpoly}
  \cite[9.10]{pzci} 
Let $X$ be a Gorenstein gci space, then $H_*(\Omega X)$ is left and
right Noetherian and it is finitely generated over a central
polynomial subalgebra. 
\end{lemma}

Since $X=BG_p$ is automatically Gorenstein, the first assumption is 
absolutely harmless. We will proceed on the assumption that $X=BG_p$
is a gci space. If $G$ is $p$-perfect, by Levi's
Dichotomy Theorem \cite{LeviDichotomy}, $\HOBG$ otherwise has at least
semi-exponential growth, and in that case  it is hard to even
understand what it means for a module to be finitely generated.

We will assume the growth is polynomial and
name the generators of the polynomial subring as follows.

\begin{assumption}
  $\HOX$ is finite over a central polynomial subalgebra $k[\tau_1, \ldots, 
\tau_s]$. 
\end{assumption}

\subsection{The commutative case}
To set up notation we recall the construction in the
commutative case. 

Given a commutative Noetherian ring $R$, and an element $\tau$ we
define the unstable Koszul complex $K_n(R)=(R\stackrel{\tau^n}\lra R)$
and  their direct limit,
the stable Koszul complex $K_{\infty} (R)=(R\lra R[1/\tau])$.

Now, given an ideal $I=(\tau_1,
\ldots, \tau_s)$ we may construct the stable Koszul complex
$$K_{\infty}(I)=K_{\infty}(\tau_1)\tensor_R \cdots \tensor_R
K_{\infty}(\tau_s). $$
One may check that up to homology isomorphism this only depends on the
radical of $I$ and for an $R$-module $N$ we may define the local
cohomology by  
$$H^*_I(R;N):=H^*(K_{\infty}(I)\tensor_R N). $$
Grothendieck showed that this calculates the right derived  functors
of the $I$-power torsion functor. Similarly we may omit the degree 0
copy of $R$ and regrade to get the Cech complex, which fits into a fibre sequence
$$K_{\infty}(R)\lra R \lra CH_I(R).$$
The Cech cohomology is defined by
$$CH^*_I(R;N)=CH_I(R)\tensor_R N, $$
and it is easy to see that it calculates the
right derived functors of global sections of $N$, viewed as a sheaf
over $\spec(R)\setminus V(I)$.

These constructions immediately adapt to the case that $R$ is a
commutative ring spectrum and $I$ is an ideal in the Noetherian ring
$R_*$. The stable Koszul complex is replaced
by the $R$-module $\fib (R\lra R[1/x])$. We write $\Gamma_I N:=K_{\infty}(R)\tensor_R
N$ and $L_IN=CH_I(R)\tensor_R N$. One may see that $\Gamma_IN$ is the
cellularization with respect to $K:=K_1(\tau_1, \ldots, \tau_s)$ and $L_IN$
is the $K$-nullification. 

The construction gives a natural
filtration giving rise to spectral sequences
$$H^*_I(R_*; N_*)\Rightarrow \pi_*(\Gamma_I N) \mbox{ and }
CH^*_I(R_*; N_*)\Rightarrow \pi_*(CH_I N). $$

When $K$ and $k$ generate the same localizing subcategory
we write $\Gamma_kN=\Gamma_IN$ and $L_kN=L_IN$. This
holds for instance if there is a proxy small map $R\lra k$
and $(R_*, I, k)$ is a local ring.

\subsection{Adapting to bimodules}
If  $\COX$ is finite as a  module over a commutative ring 
spectrum $R$ with a map $k[\tau_1, \ldots, \tau_s] \lra R_*$ then
by Lemma \ref{lem:cob}, the localization $\TCOX$, when considered as
an $R$-module can be constructed by a Cech complex. 

We now observe that we may adapt this to $R=\COX$ provided $\tau_1, \ldots , 
\tau_s$ may be realized by maps of $R$-bimodules. 

\begin{lemma}
  \label{lem:LkSS}
  Suppose $\COX$ is proxy small and  $\HOX$ is finite over a central polynomial subalgebra
  $k[\tau_1,\ldots, \tau_s]$. Provided $\tau_i\in \HOX$ is realized by a bimodule
  self-map of $\COX$, there are spectral sequences for calculating
  the homotopy of the $k$-cellularization $\Gamma_kM $ and
  $k$-nullification $L_kM$ of a $\COX$-module $M$ taking a familiar
  form. If $I=(\tau_1, \ldots, \tau_s)$ the spectral sequences are 
  $$H^*_I(\pi_*(M))\Rightarrow  \pi_*(\Gamma_k M)$$
  and 
  $$CH^*_I(\pi_*(M))\Rightarrow  \pi_*(L_k M). $$
\end{lemma}

\begin{remark}
If $s=1$ this collapses to the familiar 
calculation 
$$\pi_*(L_kM)=M_*[1/\tau]. $$
\end{remark}

\begin{proof}
We repeat the construction with $R=\COX$ and  realize $\Gamma_IR$ as an
$(R,R)$-bimodule. It is built from $K=K_1(\tau_1, \ldots, \tau_s)$,
and  $\pi_*K$ is finite dimensional over $k$, and is therefore finitely
built from $k$ as a bimodule. Thus $\Gamma_kM=\Gamma_kR\tensor_R M$ is
built from $k\tensor_RR$. It follows that $\Gamma_kM$ is $k$-cellular,
and the map $\Gamma_kM \lra M$ is a $k$-colocal equivalence. 
\end{proof}

In the light of this, it is natural to make a definition. 
\begin{defn}
  We say that $\COX$ has a {\em central system of parameters} if 
  there are central elements $\tau_1, \ldots , \tau_s\in \HOX$ over which 
  $\HOX$ is finite and these are realized by bimodule maps 
  $\hat{\tau}_i:\COX\lra \COX$. 
\end{defn}

Thus if $\tau_1, \ldots , \tau_s$ may be represented by homotopically central 
elements of $\COBG$,  we have an equivalence of modules 
  $$\TCOX\simeq CH_{\tau_1, \ldots , \tau_s}(\COX) .$$

\subsection{Hochschild cohomology}

In order to show that polynomial generators $\tau_i$ may be chosen
homotopically central,  we need control over the Hochschild cohomology
 $HH^*(\COX)$. Since
$HH^*(\COX)\cong HH^*(C^*X)$ \cite[5.3]{Dsgpq} or \cite[1.10.1]{BensonDsg}, there are two obvious  spectral sequences for this: 
$$HH^*(H^*X)\Rightarrow HH^*(C^*X) \mbox{ and }
HH^*(\HOX)\Rightarrow HH^*(\COX).$$
In the second case the edge homomorphism is a map 
$$HH^*(C^*(X))=\pi_*(\Hom_{C^*X^e}(C^*X, C^*X))\lra 
\pi_*(\Hom_{C^*(X)}(k,k))\cong \HOX. $$

\begin{example}
  In the case that $G$ has a cyclic Sylow subgroup of order $p^n$ \cite{Dsgpq}
describes both spectral sequences. The differential  in the first occurs later (at the
$(p^n-1)$-page, whilst the second occurs precisely $(p^n-1)/q$ stages
earlier, where $2q$ is the degree of the polynomial generator). 

In that case the $E_2$-pages of the two spectral sequences are
isomorphic (though graded quite differently), and the differential is
on `the same' element (that corresponding to the exterior generator of
$\HOX$). 
\end{example}

\begin{defn}
  The ring $C^*(X)$ is said to be {\em Hochchild cohomology complete
    intersection (HHci)}  if the spectral sequence
  $HH^*(H^*X)\Rightarrow HH^*(C^*X)$ collapses at a finite stage. 
  \end{defn}

\begin{prop}
  \label{prop:collapsesurvive}
 If  the spectral sequence $HH^*(H^*X)\rightarrow HH^*(C^*X)$
 collapses at the $E_{r-2}$ page, then all $p^r$th powers are
 infinite cycles. 
\end{prop}

\begin{proof}
In  characteristic $p$ every differential vanishes on $p$th powers by
the Leibniz rule, so that if the spectral sequence collapses at the
$E_{r-2}$-page,  $x^{p^r}$ survives for every element $x$. 
\end{proof}

A sufficient condition for the collapse of the Hochschild cohomology 
spectral sequence for cochains comes 
from the coefficient level. 
\begin{lemma}
  \label{lem:cciisHHci}
    If $H^*(F)$ is a complete intersection then the spectral sequence 
    $$HH^*(H^*(F))\Rightarrow HH^*(C^*F)$$
    collapses at a finite stage: cci implies HHci. 
    \end{lemma}

  \begin{proof}
If $H^*F=k[x_1, \ldots, x_c]/(f_1, \ldots , f_c)$ then the Hochschild 
cohomology $HH^*(H^*(F))$ is described explicitly by Buchweitz and 
Roberts \cite{BR}. We only need to know that there are generators in 
bidegrees $(-1, -|x_i|)$ and $(-2, -|f_i|)$. Since $F$ is finite this 
shows  that $HH^*(H^*F)$ is concentrated in 
a strip of vertical length equal to the dimension of $F$  and below a 
line of slope $1$. It follows that the spectral sequence 
$$HH^*(H^*(F))\Rightarrow HH^*(C^*F)$$
collapses at a finite stage. 
\end{proof}

  \subsection{Retreat to normalizations}

  We now suppose given a normalization of $C^*X$, which is a $p$-adic
  fibration  $F\lra X\lra Y$ with $Y$ regular and $F$ finite. We will
  show that we may obtain the required behaviour  by imposing conditions
  on $F$ rather than on $X$.
  
\begin{defn} We say that $X$ {\em has a normalization with Property P} 
  if  there  is a normalization $F\lra X\lra Y$ so that $F$ has 
  Property P. 

  The four properties P of interest are
  \begin{itemize}
    \item   having a central system of
  parameters ($C_*(\Omega F)$ has a central system of parameters),
   \item being Hochschild cohomology ci (HHci) (i.e., the spectral
  sequence $HH^*(H^*(F))\Rightarrow HH^*(C^*(F))$
  collapses at a finite stage,
  \item being coefficient ci (cci)
    (i.e., $H^*(F)$ is ci),   and
    \item being strongly spherical ci (ssci)
      (i.e.,  $H^*(F)$ is an exterior algebra).
    \end{itemize}
    It is clear that ssci implies cci, and Lemma \ref{lem:cciisHHci} shows cci
    implies HHci. 
  \end{defn}

  \begin{lemma}
    \label{lem:csopisSS}
  If $X$ has a normalization with a central system of parameters then as a module
  over $C_*(\Omega F)$,  $\TCOX$ has a Cech construction and there is
  a spectral sequence
  $$CH^*_k(\HOX)\Rightarrow \pi_*(\TCOX). $$
\end{lemma}

\begin{proof}
We have a  cofibre sequence 
$$C_*(\Omega F)\lra \COX \lra C_*\Omega Y,$$
and since $Y$ is regular, $\Omega Y$ is finite. In particular, 
$$C_*\Omega Y \simeq \COX\tensor_{C_*(\Omega F)} k$$
is finite dimensional, and hence $\COX$ is finitely built by 
$C_*(\Omega F)$.  This shows that some power of each $\tau_i\in \HOX$
lifts  to $H_*(\Omega F)$.

By the hypothesis on $F$, $\TCOX$ can be constructed as a $C_*(\Omega
F)$-module using a Cech construction and apply Lemma \ref{lem:LkSS}.
\end{proof}

\subsection{Conditions on cochains}
Finally, we show that a complete-intersection condition on $C^*X$ (cochains),
is sufficient to guarantee the existence of a central system of
parameters in $\COX$ (chains); this then gives
the required spectral sequence for $\TCOX$.

\begin{prop}
  \label{prop:HHciiscsop}
  If $C^*(X)$ has a normalization which is HHci then the normalization
  has a central system of parameters. 
\end{prop}

\begin{proof}
  If $F\lra X\lra Y$ is the HHci normalization, we saw in Proposition
  \ref{prop:collapsesurvive} that   some powers of the classes $\tau_i\in \HOX$
  lift to $H_*(\Omega F)$, and we must show that some power of these
  survive to $HH^*(C^*(\Omega F))$.
  
  By Proposition \ref{prop:collapsesurvive}, it suffices to show
the  elements $\tau_i$ are represented in the Hochschild cohomology ring
  $HH^*(H^*F)$.

  In fact we have a spectral sequence $\Ext_{H^*(F)}^{*,*}(k, k)\Rightarrow
  H_*(\Omega F)$, and we may  choose representatives of $\tau_i$ in
  the $E_2$-term. Using the map
  $$\Ext_{H^*(F)}^{*,*}(k,k)\lra \Ext_{H^*(F \times F)}^{*,*}(H^*F,H^*F)=HH^*(H^*(F))$$
  we obtain elements $\tau_i$ in the $E_2$-term of the spectral
  sequence for $HH^*(C^*(F))$.
  \end{proof}

  \begin{remark}
It is natural to ask if the condition that $X$ is gci already implies
HHci, or indeed if gci is sufficient for  the two spectral sequences for Hochschild
cohomology to collapse. 

One may also ask if the collapse of one of the spectral
sequences  at a finite stage implies the other collapses at a finite
stage. 
\end{remark}

\subsection{Returning to groups}
We return to the special case $X=BG$. In this case it is natural to
focus on normalizations coming from group homomorphisms.

\begin{defn} We say that $G$ has a normalization with property P 
  if  there  is a group homomorphism $\rho :   G\lra U$ with $U$
  regular, so that $C^*(U/G)$ has property P. 

As before, the three properties P of interest are (in order of increasing 
strength) HHci, coefficient ci (cci), and strongly spherical ci (ssci). 
  \end{defn}

  \begin{cor}
    \label{cor:HHciiswc}
If $G$ has an HHci normalization then 
$\COBG$ is  finite over a weakly central polynomial subalgebra, and
hence the coefficients may be calculated by the spectral sequence of
Lemma \ref{lem:LkSS}.
\end{cor}

\begin{remark}
The group $U$ need not be a compact connected Lie group: we only need
it to be regular.  For example
if $k$ is of characteristic $p$ and $q|p-1$, we may consider a
non-trivial split extension  $G=C_{p^n}\sdr C_q$ (up to homotopy
equivalence this is the general case of a group with cyclic Sylow
$p$-subgroup as in \cite{Dsgpq}). In this case we may let $U=S^1\sdr
C_q$, a 1-dimensional compact Lie group which behaves like a
$(2q-1)$-dimensional compact connected Lie group.
\end{remark}

\section{Examples}
\label{sec:examples}
Methods in previous work have only applied to hypersurfaces. The
present paper 
shows that the class of groups they are
effective for is closed under products. Accordingly we can immediately generate
many examples not covered previously by using products of
s-hypersurfaces.

\begin{example} Let $G=C\sdr D$ with cyclic Sylow
$p$-subgroup $C=C_{p^n}$ and $D=C_q$ acting non-trivially. We can
express this as a strongly spherical
hypersurface via the fibration
$$T/C\lra BG\lra BU$$
where $BG=BC^{hD}$ and $BU=BT^{hD}$. This has
$$\HBG=k[X]\tensor \Lambda (T)
\mbox{ and }\HOBG=\Lambda [\xi ]\tensor k[\tau]$$
with
$$|X|=-2q, |T|=-2q+1, |\xi|=2q-1, |\tau|=2q-2.$$
\end{example}

\begin{example}
We can take $G=A_4$ with $p=2$ and $U=SO(3)$ and use the 2-adic
fibration
$$S^3\lra BA_4\lra BSO(3) . $$
\end{example}

Any product of of these examples give another ssci group, by virtue
of the fibration
$$S^{n_1}\times \cdots \times S^{n_c}\lra B(G_1\times\cdots \times
G_c)\lra B(U_1\times \cdots \times U_c).$$
One may then hope to construct  indecomposable examples from
these. For example, if $D$ is a group of order prime to $p$ and it acts  on
$U_1\times \cdots \times U_c$ preserving $G_1\times \cdots \times G_c$
in such a way that the action on $U_1/G_1 \times \cdots \times U_c/G_c$ is
trivial on homology,  then taking semidirect products gives another example
$$G=(G_1\times \cdots \times G_c)\sdr D, U=(U_1\times \cdots \times
U_c)\sdr D. $$

The following class of examples may seem more compelling. 
\begin{example}
 Suppose $\Gamma$ is a simply connected compact Lie group for which $p$ is not a
torsion prime (for example any group $SU(n)$ qualifies, and the
situation is summarised in \cite[Section 5]{BGformal}).  The
classifying space of the finite
Chevalley groups $G=\Gamma (q)$ for $q$ prime to $p$
fits into a $p$-adic homotopy pullback square
$$\xymatrix{
  BG \rto \dto &B\Gamma\dto\\
  B\Gamma \rto^(0.4){ \{1, \Psi^q\} } &B\Gamma \times B\Gamma
  }$$

Accordingly, there is a fibration $\Gamma \lra BG\lra
B\Gamma$. Since $\Gamma $ is connected and the fibre is finite, this
is a normalization, and since $p$ is not a torsion prime $H^*(\Gamma ;
\F_p)$ is exterior. This shows that $G$ is ssci and hence HHci.
\end{example}

\end{document}